\documentclass[reqno, 12pt]{amsart}
\usepackage{a4wide}
\usepackage[utf8]{inputenc}
\usepackage{amsfonts, amssymb, amsmath, amsthm, mathrsfs, amscd, enumerate, enumitem, setspace, url, color, graphicx, blindtext, scrextend}
\usepackage[all]{xy} %drawing diagrams

\usepackage{hyperref}
\definecolor{darkred}{rgb}{0.5,0,0}
\definecolor{darkgreen}{rgb}{0,0.5,0}
\definecolor{darkblue}{rgb}{0,0,0.7}
\hypersetup{linkcolor=darkblue,filecolor=darkgreen,urlcolor=darkred,citecolor=darkblue}

\theoremstyle{plain}

\newtheorem{theorem}{Theorem}[section]
\newtheorem{lemma}[theorem]{Lemma}
\newtheorem{corollary}[theorem]{Corollary}

\newtheorem{proposition}[theorem]{Proposition}

\theoremstyle{definition}

\theoremstyle{remark}
\newtheorem{remark}[theorem]{Remark}

\newtheorem*{acknowledgements}{Acknowledgements}

\numberwithin{equation}{section}
\renewcommand{\(}{\left(}
\renewcommand{\)}{\right)}

\newcommand{\NN}{\mathbb{N}}
\newcommand{\ZZ}{\mathbb{Z}}
\newcommand{\QQ}{\mathbb{Q}}
\newcommand{\Zp}{\mathbb{Z}_p}
\newcommand{\Zq}{\mathbb{Z}_q}
\newcommand{\Qp}{\mathbb{Q}_p}

\newcommand{\RR}{\mathbb{R}}

\newcommand{\ve}{\varepsilon}

\renewcommand{\AA}{\mathbb{A}}
\newcommand{\PP}{\mathbb{P}}

\renewcommand{\P}{\mathcal{P}}
\newcommand{\Q}{\mathcal{Q}}

\newcommand{\A}{\mathbf{A}}
\newcommand{\x}{\mathbf{x}}

\newcommand{\C}{\mathbf{C}}
\newcommand{\X}{\mathbf{X}}

\DeclareMathOperator{\Br}{Br}
\DeclareMathOperator{\inv}{inv}
\DeclareMathOperator{\ev}{ev}
\DeclareMathOperator{\loc}{loc}

\DeclareMathOperator{\re}{Re}

\DeclareMathOperator{\val2}{val_2}

\DeclareMathOperator{\valp}{v_{\it p}}
\DeclareMathOperator{\disc}{disc}

\DeclareMathOperator{\He}{H}

\DeclareMathOperator{\Res}{Res}

\newcommand{\dmid}{\parallel}

\begin{document}
\title{Failures of the Integral Hasse Principle for Affine Quadric Surfaces}
\author{V. Mitankin}
\address{School of Mathematics \\
University of Bristol \\
Bristol \\
BS8 1TW \\
UK
}
\email{v.mitankin@bristol.ac.uk}
\date{April, 2016}

\begin{abstract}
	Quadric hypersurfaces are well-known to satisfy the Hasse principle. However, this is no longer true in the case of the Hasse principle for integral points, where counter-examples are known to exist in dimension 1 and 2. This work explores the frequency that such counter-examples arise in a family of affine quadric surfaces defined over the integers.
\end{abstract}

\maketitle
%\setcounter{tocdepth}{1}
%\tableofcontents

\section{Introduction}
		
Given a polynomial $f \in \ZZ[x_1, \dots, x_s]$ and an integer $n$, Hilbert's 10th problem asks for an algorithm which determines if the Diophantine equation
\begin{equation} \label{maineq}
	f(x_1, \dots, x_s) 
	= 
	n
\end{equation} 	
has solution in integer $(x_1, \dots, x_s)$ or not. Although such an algorithm is known not to exist for general $f$ and $n$, we know that if $\X$ is the variety defined by the above equation over $\ZZ$ and $X = \X \times_\ZZ \QQ$, then a necessary condition for \eqref{maineq} to be soluble over $\ZZ$ is that
\begin{equation*} 
	\X(\A) 
	= 
	X(\RR) \times \prod_{p} \X(\Zp) \neq \emptyset,
\end{equation*}
where $\ZZ_p$ is the ring of $p$-adic integers. In general, it is easier to determine if $\X(\A)$ is non-empty and we will refer to it as the set of ad\`{e}lic points on $\X$. If $f$ and $n$ are of a special shape it might be the case that the existence of an ad\`{e}lic point is also a sufficient condition for the existence of an integer point on $\X$. Thus it seems natural to raise the question of whether $\X(\A) \neq \emptyset$ implies $\X(\ZZ) \neq \emptyset$? We say that $\X$ fails the integral Hasse principle if the answer to this question is negative. This subject is classical when $f$ is an integral quadratic form, which we will assume from now on. It is an old result that one can always effectively determine if $\X$ has an integer point when $f$ is definite. On the other hand, when $f$ is a non-singular indefinite quadratic form and $n$ is non-zero, the integral Hasse principle holds if $s \ge 4$ \cite[Thm~1.5, p 131]{C08}. However, the latter statement is false infinitely often when $s = 3$ \cite[p. 324, Example 1.2]{SPX04}. For example, let
\begin{equation} \label{IHP failure example}
	\X: \quad a^2 x^2 + b^{2m} y^2 - b z^2 = 1,
\end{equation}
with $a, b$ and $m$ positive integers such that $(a, b) = 1$, with $2 \mid a$ and $b \equiv 5$ mod $8$. Then $\X$ fails the integral Hasse principle (cf. \cite[\S 8.2]{CTX09}).
	
The purpose of the current paper is to estimate how frequently such counter-examples appear when $f$ ranges over diagonal indefinite ternary quadratic forms and to compare this amount to the number of surfaces in the family which have a non-empty set of ad\`{e}lic points. Our main tool is the integral Brauer--Manin obstruction, which is defined in $\S\ref{IBM}$ and which was introduced by Colliot-Th\'el\`ene and Xu \cite[\S 1]{CTX09}. They show that the integral Brauer--Manin obstruction is equivalent to the classical one given by the spinor genus of a quadratic form \cite[\S 7.2]{CTX09}. Moreover, in the case of $\X$ given by an indefinite ternary quadratic form $f$ it is shown to be the only obstruction to the integral Hasse principle \cite[\S 7.3]{CTX09}. 
	
In the last few years many questions of similar essence have been the object of investigation. In \cite{BB14a}, de~la~Bret\`{e}che and Browning have looked at Ch\^{a}telet surfaces over $\QQ$, obtaining an asymptotic formula for the density of such surfaces which fail the usual Hasse principle for rational points. The same has been done for families of coflasque tori which fail the Hasse principle \cite[Thm~1.1]{BB14b}. More recently, Jahnel and Schindler \cite[Thm~1.2]{JS15} have established an asymptotic formula for the number of del Pezzo surfaces of degree four in a Birch and Swinnerton-Dyer type family which fail the Hasse principle. Our work appears to be the first time that these questions have been addressed in the context of integral points on affine varieties.

Throughout the rest of the paper for given non-zero integers $a$, $b$, $c$, $n$, with $a$, $b$, $c$ not all the same sign, our primary objects of study are the quadric surfaces $\X_{a, b, c} \subseteq \AA^3$ defined by
\begin{equation} \label{X}
	\X_{a, b, c}: \quad a x^2 + b y^2 + c z^2 = n.
\end{equation}
For any fixed non-zero $n \in \ZZ$ let $\mathcal{F}_n$ be the family of quadric surfaces $\X_{a, b, c}$ defined as in \eqref{X}. We order the surfaces in $\mathcal{F}_n$ with respect to the usual height function which is given by $\He(\X_{a, b, c}) = \max\{|a|, |b|, |c|\}$.

Firstly, we would like to study how many surfaces in the family have points everywhere locally. For any real $B \ge 1$ we introduce a quantity which measures the frequency that such surfaces of height not exceeding $B$ appear when we range over the family $\mathcal{F}_n$. It is given by
\begin{equation*}
	N_{\loc}(B)
	= 
	\# \left\{ \X_{a, b, c} \in \mathcal{F}_n \ : \ \He(\X_{a, b, c}) \le B \mbox{ and } \X_{a, b, c}(\A) \neq \emptyset \right\}.
\end{equation*}
Our first result confirms that a positive proportion of surfaces in the family $\mathcal{F}_n$ have points everywhere locally and provides an explicit interpretation of their density.

\begin{theorem} \label{Th0}
	For any non-zero $n \in \ZZ$ the limit $\lim_{B \to \infty} B^{-3} N_{\loc}(B)$ exists, it is non-zero and it is given as a product of local densities 
	\begin{equation*}
		\lim_{B \to \infty}\frac{N_{\loc}(B)}{B^3} 
		= 
		\sigma_{\infty}
		\prod_{p} \sigma_p,
	\end{equation*}
	where $\sigma_{\infty}$ and $\sigma_p$ are defined in \eqref{Omega} and \eqref{localDensDef}. Furthermore, $\sigma_{\infty} = 1$ and for any prime $p \nmid n$ we have
	\begin{equation*}
 		\sigma_p 
 		= 
 		\begin{cases}
 			1 - \frac{3}{2}p^{-2} + \frac{1}{2}p^{-3} 
 			&\mbox{if $p > 2$}, 
 			\\
 			\frac{357}{512} 
 			&\mbox{if $p = 2$}.  
 		\end{cases}
	\end{equation*}
\end{theorem}

Bhargava, Cremona, Fisher, Jones and Keating \cite{BCFJK16} have an analogous result for homogeneous quadric hypersurfaces of any dimension. They show that the density of such hypersurfaces of fixed dimension which have a $\ZZ$-point is given as a product of local densities and they compute these densities explicitly with respect to the Gaussian Orthogonal Ensemble distribution. In our setting, the density depends intimately on the value of $n$. A method for computing $\sigma_p$ for $p \mid n$ is given in $\S\ref{LD}$. 
	
Our next result is an upper bound for the number of surfaces in the family $\mathcal{F}_n$ which admit a Brauer--Manin obstruction to the integral Hasse principle. To do so, for any real $B \ge 1$, we define
\begin{equation*} 
	N_{\Br}(B) 
	= 
	\# \left\{ \X_{a, b, c} \in \mathcal{F}_n \ : \ \He(\X_{a, b, c}) \le B,\ \X_{a, b, c}(\A) \neq \emptyset \ \mbox{and} \ \X_{a, b, c}(\ZZ) = \emptyset \right\}.
\end{equation*}
Then we have the following result.

\begin{theorem}\label{Th1}
	For any non-zero $n \in \ZZ$ we have
	\begin{equation*}
		N_{\Br}(B) 
		\ll_n 
		B^{\frac{3}{2}} \( \log B \)^3.
	\end{equation*}
\end{theorem}

It is worth emphasising that the number of surfaces in $\mathcal{F}_n$ with non-trivial Brauer group is of order $B^3$ (see \cite[\S 5.6]{CTX09}). Theorem~\ref{Th1} therefore shows that although most of the time there is a non-trivial Brauer group, the existence of a Brauer--Manin obstruction to the integral Hasse principle is quite a rare phenomenon. Our proof is based on ideas found in the work of Colliot-Th\'{e}l\`{e}ne, Kanevsky and Sansuc \cite{CKS87}, as developed further by Bright, Browning and Loughran \cite[\S 6]{BBL16}. In fact, a minor modification of our argument is enough to show that there are just $O(B^{\frac{3}{2}} \( \log B \)^3)$ surfaces in the family which satisfy strong approximation away from $\infty$. Moreover, if we compare Theorem~\ref{Th1} with Theorem~\ref{Th0} we see that $100\%$ of the surfaces in $\mathcal{F}_n$ satisfy the integral Hasse principle (and  yet fail strong approximation away from $\infty$).
	
It is a challenging question to understand the actual asymptotic behaviour of $N_{\Br}(B)$ when $B$ tends to infinity. With our technique, in order to verify for which quadric surfaces $\X_{a, b, c} \in \mathcal{F}_n$ the Brauer--Manin set is empty we rely on finding an explicit uniform generator for $\Br X_{a, b, c} / \Br \QQ$ across the whole family. An algorithm for finding an explicit generator of $\Br X_{a, b, c} / \Br \QQ$ is found in \cite[\S 5.8]{CTX09}, provided we know a $\QQ$-rational point on $X_{a, b, c}$. Although the Hasse--Minkowski theorem guarantees such a rational point exists for each surface in $\mathcal{F}_n$ with non-empty set of ad\`{e}les, recent work of Uematsu \cite{U16} shows that a uniform generator for the family $\mathcal{F}_n$ does not exist. 
	
The example considered in \eqref{IHP failure example} shows that $N_{\Br}(B) \gg B$ when $n = 1$. Our next goal is to provide a lower bound for $N_{\Br}(B)$ which shows that the upper bound in Theorem~\ref{Th1} is not far from the truth. For $a$, $b$, $c \in \NN$ not divisible by $3$ we consider the subfamily $\mathcal{F}_1'$ of $\mathcal{F}_1$ given by the surfaces
\begin{equation} \label{X'}
	\X'_{a, b, c}: \quad 9a x^2 - 3b y^2 + 16c^2 z^2 = 1,
\end{equation}
with $a$, $b$ squarefree, $a \equiv 1$ mod $8$ and $(ab, 2c) = 1$. For this family $(0, 0, \pm 1/4c) \in \X'_{a, b, c}(\QQ)$.  For any real $B \ge 1$ we define
\begin{equation*}
	N'_{\Br}(B)
	= 
	\# \left\{ \X'_{a, b, c} \in \mathcal{F}_1' \ : \ \He(\X'_{a, b, c}) \le B, \ \X'_{a, b, c}(\A) \neq \emptyset \mbox{ and } \X'_{a, b, c}(\ZZ) = \emptyset \right\}. 
\end{equation*}
Then an asymptotic formula for $N'_{\Br}(B)$ can be given.

\begin{theorem} \label{Th2}
	We have
	\begin{equation*}
		N'_{\Br}(B)
		=
		C \frac{B^\frac{3}{2}}{\(\log B\)^{\frac{1}{2}}}
		+
 		O\( \frac{B^\frac{3}{2} \( \log\log B \)^{\frac{3}{2}}}{\log B}\).
	\end{equation*} 
	Here, if $\chi$ is the extension of the quadratic character $\( \frac{3}{\ast}\)$ mod $24$, then
	\begin{equation*}
		C
		=
		\frac{1}{270 \pi^\frac{1}{2}}
		\prod_{p}\( 1 + \frac{1}{2p} \)\( 1 - \frac{1}{p} \)^{1/2} \( 1 + \frac{\chi(p)}{2p + 1} \) \(1 - \frac{1 + \chi(p)}{p(2p + 1 + \chi(p))} \).
	\end{equation*}
\end{theorem}

It is clear that if we multiply the left and right hand side of \eqref{X'} by a non-zero $n \in \ZZ$ we get an asymptotic formula for the number of surfaces in a subfamily of $\mathcal{F}_n$ which fail the integral Hasse principle. Therefore, we have the following result which shows that the upper bound in Theorem~\ref{Th1} is optimal up to $(\log B)^{\frac{7}{2}} $.

\begin{corollary}
	For every non-zero $n \in \ZZ$ we have
	\begin{equation*}
		N_{\Br}(B) 
		\gg_n 
		\frac{B^{\frac{3}{2}}}{\(\log B \)^{\frac{1}{2}}}.
	\end{equation*}
\end{corollary}

This paper is organised as follows. In $\S\ref{IBM}$ we give an overview of the integral Brauer--Manin obstruction and present the way we are going to study surfaces in the families we are interested in. We end $\S\ref{IBM}$ with the statement of Proposition~\ref{Proposition} which plays a key role in the proof of Theorem~\ref{Th1}. Section $\ref{LD}$ consists of the proof of Theorem~\ref{Th0}. In $\S\ref{UB}$ we prove Proposition~\ref{Proposition} and Theorem~\ref{Th1}. Finally, $\S\ref{AF}$ is dedicated to the proof of Theorem~\ref{Th2}.
	
\begin{acknowledgements} 
	The author would like to express his deepest gratitude to Tim Browning for suggesting two of these problems and for the nearly constant support both in the mathematical sense and in improving the script. The author is also very grateful to Damaris Schindler for suggesting the problem in Theorem~\ref{Th0} and to Dan Loughran and Damaris Schindler for many helpful suggestions. The author would like to thank Jean-Louis Colliot-Th\'{e}l\`{e}ne, Roger Heath-Brown, Adam Morgan, Efthymios Sofos and Vinay Viswanathan for helpful discussions and remarks. The author would also like to thank the anonymous referee for giving alternative short proof of Proposition~\ref{Proposition} and for the useful comments and suggestions. While working on this paper the author was
	supported by ERC grant \texttt{306457}.
\end{acknowledgements}

\section{The integral Brauer--Manin obstruction}\label{IBM}

In our setting $\X$ will always be a separated scheme of finite type over $\ZZ$. For this reason and in order to simplify the exposition we shall explain the construction of the integral Brauer--Manin obstruction in this setting only. This is done as a summary of the construction that was originally developed by Colliot-Th\'el\`ene and Xu in \cite[\S 1]{CTX09}. Let $X = \X \times_\ZZ \QQ$. The standard notation $\Qp$ will be used for the field of $p$-adic numbers and we shall set $\QQ_\infty = \RR$. For each finite prime $\Zp$ denotes the ring of $p$-adic integers and at $p = \infty$ we adopt the convention that $\ZZ_\infty = \RR$.

For each element $\alpha$ of the Brauer group $\Br X$ of $X$ there is a map $\ev_{\alpha, p} : X(\Qp) \longrightarrow  \Br \Qp$ which evaluates this element at a given $p$-adic point on $X$ and for each prime $p$ including the infinite one there is an embedding $\inv_p : \Br \Qp \hookrightarrow \QQ / \ZZ$. Thus we get the composition of maps
\begin{equation} \label{eq:ev composed with inv}
	X(\Qp) 
	\xrightarrow{\ev_{\alpha, p}}  
	\Br \Qp 
	\xrightarrow{\inv_p} 
	\QQ / \ZZ.
\end{equation}
An ad\`{e}le on the $\QQ$-variety $X$ is a point $\{ \x_p \} \in \prod_{p \le \infty} X(\Qp)$ such that, for all but finitely many primes $\x_p$ belongs to $\X(\Zp)$. The set of ad\`{e}les of $X$ will be denoted by $X (\A_\QQ)$. The set of rational points on $X$ can be diagonally embedded in $X(\A_\QQ)$ and we will use $X(\QQ)$ also for the image of $X(\QQ)$ in $X(\A_\QQ)$ under this embedding. There exists a natural pairing, the Brauer--Manin pairing, between the set of ad\`{e}les $X (\A_\QQ)$ and the Brauer group $\Br X$ of $X$. This pairing is given by
\begin{equation*}
	\begin{split}
		X(\A_\QQ) \times \Br X 
		&\rightarrow 
		\QQ / \ZZ, 
		\\
		\(\{ \x_p \}, \alpha \) 
		&\mapsto 
		\sum_{p \le \infty} \inv_p \( \ev_{\alpha, p} \( \x_p \)\).
	\end{split}
\end{equation*}
It is known that any element of $X(\QQ)$ is in the left kernel of the above pairing and the image of $\Br \QQ \rightarrow \Br X$ is in the right kernel of that pairing. Since we are interested in the existence of integer points on $\X$ we recall that the set of ad\`{e}les on $\X$ is defined as
\begin{equation*}
	\X(\A) = X(\RR) \times \prod_{p} \X(\Zp)	
\end{equation*}
and one can view $\X(\A)$ as a subset of $X(\A_\QQ)$ by separateness of $\X$. Therefore, the described pairing above induces a pairing
\begin{equation*}
	\X(\A) \times \(\Br X / \Br \QQ\)
	\rightarrow 
	\QQ / \ZZ
\end{equation*}
which vanishes on the image of $\X(\ZZ)$ in $\X(\A)$. As before, we will use $\X(\ZZ)$ also for the image of $\X(\ZZ)$ in $\X(\A)$. We define the Brauer--Manin set $\X(\A)^{\Br X}$ to be the left kernel of the induced pairing. Thus, it follows that we have the inclusions 
\begin{equation*}
	\X(\ZZ)
	\subseteq 
	\X(\A)^{\Br X} 
	\subseteq 
	\X(\A).
\end{equation*}
We say that there is an integral Brauer--Manin obstruction for $\X$ or a Brauer--Manin obstruction to the existence of integer points on $\X$ if the set of  ad\`{e}les $\X(\A)$ is non-empty but the Brauer--Manin set $\X(\A)^{\Br X}$ is empty. To avoid any confusion, we remark that in the present work when we refer to the Brauer--Manin obstruction we will always mean the Brauer--Manin obstruction to the integral Hasse principle.

For any non-zero rational number $n$ and any non-singular quadratic form $f(x, y, z)$ with rational coefficients let $X \subseteq \AA^3$ denote the quadric surface defined over $\QQ$ by 
\begin{equation*}
	X: \quad f(x, y, z) = n.
\end{equation*}
Let $d = - n \disc (f)$. If $d$ is a square in $\QQ$, then $\Br X / \Br \QQ = 0$ \cite[\S 5.6]{CTX09} and therefore there is no integral Brauer--Manin obstruction. Assume now that $d$ is not a square in $\QQ$ and $X (\QQ) \neq \emptyset$. Then $\Br X / \Br \QQ \cong \ZZ / 2\ZZ$ and there exists a way to find an explicit non-trivial generator of this group (cf. \cite[\S 5.8]{CTX09}). We will briefly explain the algorithm giving this generator. Let $Y \subseteq \PP^3$ be the smooth projective quadric given by the homogeneous equation
\begin{equation*}
	f(x, y, z) - n t^2 
	= 
	0.
\end{equation*}
When $X (\QQ)$ is non-empty there is a $\QQ$-rational point $M$ of $Y(\QQ)$ with $t \neq 0$. Let $l_1 (x, y, z, t)$ be a linear form with coefficients in $\QQ$ defining the tangent plane to $Y$ at $M$. There exist linear forms $l_2, l_3, l_4$ in $x, y, z, t$ and a constant $c \in \QQ^{\ast}$ such that
\begin{equation*}
	f(x, y, z) - nt^2 
	= 
	l_1 l_2 + c (l_3^2 - d l_4^2).
\end{equation*}
Take $\rho = l_1(x, y, z, t) / t$ and $\sigma = l_2(x, y, z, t) / t$. Then $\alpha = (\rho, d)$ is a non-trivial generator of $\Br X / \Br \QQ$ and $(\rho, d) = (-c \sigma, d)$ as elements of that group. 

With the above assumptions, when $\Br X / \Br \QQ \cong \ZZ / 2\ZZ$, in order to check if the Brauer--Manin set is empty we will work with the generator $\alpha$ of $\Br X / \Br \QQ$. For simplicity if $\x_p \in \X(\Zp)$ we will use $\alpha\(\x_p\)$ for the image of $\x_p$ under the map $\ev_{\alpha, p}$. In this case, the composition $\inv_p \circ \ev_{\alpha, p}$ is given by 
\begin{equation*}
	\inv_p \alpha\(\x_p\)
	= 
	\begin{cases}
		0 
		&\mbox{ if } \alpha\(\x_p\) \mbox{ is split over } \Qp, 
		\\
		1/2 
		&\mbox{ if } \alpha\(\x_p\) \mbox{ is not split over } \Qp.
	\end{cases}
\end{equation*}	
Therefore, if $a, b, c, n$ are non-zero integers and $\X_{a, b, c}$ is as in \eqref{X} with $\X_{a, b, c}(\A) \neq \emptyset$, then an integral Brauer--Manin obstruction for $\X_{a, b, c}$ exists precisely when 
\begin{equation*}
	\sum_{p \le \infty} \inv_p (\alpha(\x_p)) 
	= 
	\frac{1}{2}
\end{equation*}
for each point $\{\x_p\} \in \X_{a, b, c}(\A)$.

\begin{remark}\label{SurjRemark}
	Assume that the map $\inv_q \circ \ev_{\alpha, q} : \X(\Zp) \rightarrow \{0, 1/2\}$ is surjective at the prime $q$, say. One can take an ad\`{e}lic point $\{\x_p\} \in \X_{a, b, c}(\A)$ for which the above sum is equal to $1/2$ and change $\x_q$ with $\x_q' \in \X_{a, b, c}(\Zq)$ for which $\inv_q \circ \ev_{\alpha, q}$ takes the other value. In this way another ad\`{e}lic point on $\X_{a, b, c}$ is obtained for which the above sum is equal to $0$. Therefore, there is no integral Brauer--Manin obstruction for $\X_{a, b, c}$. We will use this observation in the proof of Proposition~\ref{Proposition}. 
\end{remark}

We end this section by recording a key result needed for the proof of Theorem~\ref{Th1}.

\begin{proposition} \label{Proposition}
	Let $a$, $b$, $c$ and $n$ be non-zero integers and define $\X_{a, b, c}$ as in \eqref{X}. Then there is no Brauer--Manin obstruction to the existence of integer points on $\X_{a, b, c}$ whenever there exists a prime $p \neq 2$ and an odd positive integer $v$ such that any of the following occur:
	\begin{enumerate}
		\item[\emph{(i)}]{ $p^v \dmid a$ and $p \nmid bcn$,
		}
		\item[\emph{(ii)}]{ $p^v \dmid b$ and $p \nmid acn$,
		}
		\item[\emph{(iii)}]{ $p^v \dmid c$ and $p \nmid abn$.
		}
	\end{enumerate}
\end{proposition}

The proof of this result will be given in $\S\ref{UB}$.

\section{Proof of Theorem~\ref{Th0}} \label{LD}
	
During this section $n$ will be a fixed non-zero integer and we determine $\X_{a, b ,c, n} \subseteq \AA^3$ by \eqref{X} but we allow the coefficients $a, b, c$ to arbitrary elements of the larger ring $\Zp$. We shall use $X_{a, b, c, n}$ for the surface when $a, b, c$ are real numbers. Note that in \eqref{X} we required that two of $a$, $b$, $c$ to have different signs. Let
\begin{equation*}
	I 
	= 
	\left\{ (a, b, c) \in [-1, 1]^3 \ : \ abc \neq 0 \mbox{ and $a$, $b$, $c$ not all the same sign} \right\}.
\end{equation*}
Then it is clear that for each $(a, b, c) \in I$ we have $\ X_{a, b, c, n}(\RR) \neq \emptyset$. Let
\begin{equation} \label{Omega}
	\begin{split}
		\Omega_p
		&= 
		\left\{ (a, b, c) \in \Zp^3 \ : \ \X_{a, b, c, n}(\Zp) \neq \emptyset \right\}, 
		\\
		\Omega_\infty
		&= 
		\left\{ (a, b, c) \in I \ : \ X_{a, b, c, n}(\RR) \neq \emptyset \right\} = I.
	\end{split}
\end{equation}
For any prime $p$ let $\mu_p$ be the normalised Haar measure on $\Zp^3$ such that $\mu_p\( \Zp^3 \) = 1$ and let $\mu_{\infty}$ be the Lebesgue measure on $\RR^3$. Then we define the local densities as
\begin{equation} \label{localDensDef}
	\sigma_p
	=
	\mu_p\(\Omega_p\), 
	\quad
	\sigma_{\infty}
	= 
	\frac{\mu_{\infty}\(\Omega_\infty\)}{\mu_{\infty}\( I \)}
	= 
	1.
\end{equation} 	

We start by showing that the explicit expressions for the local densities $\sigma_p$ and $\sigma_\infty$ in the second part of Theorem~\ref{Th0} are the correct ones.

\begin{lemma} \label{LocDens}
	We have $\sigma_\infty = 1$ and 
 	\begin{equation*}
		\sigma_p 
		= 
		\begin{cases}
			1 - \frac{3}{2}p^{-2} + \frac{1}{2}p^{-3} 
			&\mbox{if $p \nmid 2n$}, 
			\\
			\frac{357}{512} 
			&\mbox{if $p = 2, 2 \nmid n$}.  
		\end{cases}
	\end{equation*}
\end{lemma}

\begin{proof}
	We have already seen that $\sigma_\infty = 1$. We will complete the proof of this result by giving an algorithm for computing $\sigma_p$ for any odd prime $p$, and for $p = 2$ when $n$ is odd. Let $p$ be any fixed prime and for each $j$ such that $p^j \mid n$ let $v_j = \valp(n p^{-j})$ and
	\begin{equation*}
		\begin{split}
			\kappa_p(j) 
			&= 
			\mu_p \( \left\{ (a, b, c) \in \Zp^3 \ : \ \X_{a, b, c, n p^{-j}} (\Zp) \neq \emptyset \mbox{ and } \min\{ \valp(a), \valp(b), \valp(c) \} = 0 \right\} \), 
			\\
			\kappa_p^c(j) 
			&= 
			\mu_p \( \left\{ (a, b, c) \in \Zp^3 \ : \ \X_{a, b, c, n p^{-j}} (\Zp) = \emptyset \mbox{ and } \min\{ \valp(a), \valp(b), \valp(c) \} = 0 \right\} \).
		\end{split}
	\end{equation*}
	Then, since $\X_{a, b, c, n} (\Zp) \neq \emptyset$ implies that $\min\{ \valp(a), \valp(b), \valp(c) \} \le \valp(n)$ we can decompose $\sigma_p$ as
	\begin{equation} \label{sigma_p decomp_0}
		\sigma_p 
		=
		\sum_{0 \le j \le \valp(n)} \frac{1}{p^{3j}} \kappa_p(j).
	\end{equation}
	It is clear that for each $j$ we have $\kappa_p(j) = 1 - p^{-3} - \kappa_p^c(j)$. Thus, we get
	\begin{equation} \label{sigma_p decomp}
		\sigma_p
		=
		1 - \frac{1}{p^{3\valp(n) + 3}} - \sum_{0 \le j \le \valp(n)} \frac{1}{p^{3j}} \kappa_p^c(j). 
	\end{equation}
	Consider the case of an odd prime $p$. Firstly, we compute $\kappa_p^c(\valp(n))$. If $\valp(a) = 0$, then $a x^2 + b y^2 + c z^2 = n p^{-\valp(n)}$ is soluble over $\Zp$ except when $\valp(b)$ and $\valp(c)$ are both positive and $an p^{-\valp(n)}$ is not a square mod $p$. Thus, by the symmetry in $a$, $b$, $c$ in the definition of $\kappa_p^c(j)$ we get
	\begin{equation} \label{kappa_p}
		\kappa_p^c(\valp(n))
		= 
		\frac{3(p - 1)}{2p^3}.
	\end{equation}
	If $p \mid n$ we proceed with the remaining terms $\kappa_p^c(j)$ one by one in a similar manner. Again, we can assume that $\valp(a) = 0$. The conditions for solubility of $a x^2 + b y^2 + c z^2 = n p^{-j}$ over $\Zp$ depend on whether each of $v_j$, $\valp(b)$ and $\valp(c)$ is even or odd and on the values of $\valp(b)$ and $\valp(c)$ compared to $v_j$. To compute $\kappa_p^c(j)$ we look separately at each of the cases mentioned above up to a  symmetry. The precise conditions in each of them are:
	\begin{enumerate}
		\item[---]{Suppose that $\valp(b) > v_j$ and $\valp(c) > v_j$.
 		\begin{itemize}
			\item{If $v_j$ is even, then there are no $\Zp$-points on $\X_{a, b, c, n p^{-j}}$ if and only if
  			\begin{equation*}
  				\(\frac{anp^{-\valp(n)}}{p}\)
  				=
  				-1.
  			\end{equation*}
 			}
			\item{If $v_j$ is odd, then there are no $\Zp$-points on $\X_{a, b, c, n p^{-j}}$.
			}
		\end{itemize}
		}	 
		\item[---]{Suppose that $\valp(b) \le v_j$ and $\valp(c) > v_j$.
		\begin{itemize}
			\item{If $v_j$ and $\valp(b)$ are even, then there is always a $\Zp$-point on $\X_{a, b, c, n p^{-j}}$.
			}
			\item{If $v_j$ is even and $\valp(b)$ is odd, then there are no $\Zp$-points if and only if 
  			\begin{equation*}
  				\(\frac{anp^{-\valp(n)}}{p}\)
  				=
  				-1.
			\end{equation*}
			}
			\item{If $v_j$ is odd and $\valp(b)$ is even, then there are no $\Zp$-points if and only if
			\begin{equation*}
				\(\frac{-abp^{-\valp(b)}}{p}\)
				=
				-1.
			\end{equation*}
			}
			\item{If $v_j$ and $\valp(b)$ are both odd, then there are no $\Zp$-points if and only if 
			\begin{equation*}
				\(\frac{bnp^{-\valp(bn)}}{p}\)
				=
				-1.
			\end{equation*}
			}
		\end{itemize}
		}
		\item[---]{Suppose that $\valp(b) \le v_j$ and $\valp(c) \le v_j$.
 		\begin{itemize}
			\item{If $v_j$ is even and $\valp(b)$ or $\valp(c)$ is even, then there is always a $\Zp$-point on $\X_{a, b, c, n p^{-j}}$.
			}
			\item{If $v_j$ is even, $\valp(b)$ and $\valp(c)$ are both odd, then there are no $\Zp$-points on $\X_{a, b, c, n p^{-j}}$ if and only if 
			\begin{equation*}
				\(\frac{anp^{- \valp(n)}}{p}\)
				=
				\(\frac{-bcp^{- \valp(bc)}}{p}\)
				=
				-1.
			\end{equation*}  
			}
			\item{If $v_j$ is odd and $\valp(b)$ and $\valp(c)$ are both even or odd, then there is always a $\Zp$-point on $\X_{a, b, c, n p^{-j}}$.
			}
 			\item{If $v_j$ is odd, $\valp(b)$ is odd and $\valp(c)$ is even, then there are no $\Zp$-points on $\X_{a, b, c, n p^{-j}}$ if and only if 
  			\begin{equation*}
  				\(\frac{bnp^{- \valp(bn)}}{p}\)
  				=
  				\(\frac{-acp^{- \valp(c)}}{p}\)
  				=
  				-1.
  			\end{equation*}  
 			}
 		\end{itemize}
		}
	\end{enumerate}
	In order to illustrate the algorithm we will demonstrate how to compute the contribution in $\kappa_p^c(j)$ coming from the case when $v_j$ is even, $\valp(b)$ and $\valp(c)$ are both odd and bounded above by $v_j$. We need to sum the measures of subsets of $\Zp^3$ consisting of triples $(a, b, c)$ for which $\valp (a) = 0$, $\valp(b)$ and $\valp(c)$ are odd positive integers not exceeding $\valp(n) - 1$, and which satisfy the conditions of the case. This sum is equal to
	\begin{equation*}
		\frac{p - 1}{p} 
		\cdot
		\( \sum_{1 \le l \le v_j/2} \frac{p - 1}{p^{2l}} \) 
		\cdot 
		\( \sum_{1 \le m \le v_j/2} \frac{p - 1}{p^{2m}} \) 
		\cdot
		\frac{1}{4}
		= 
		\frac{( p - 1) (p^{v_j} - 1)^2}{4 (p + 1)^2 p^{2 v_j + 1}}.
	\end{equation*} 
	What is left is to multiply the above by $3$. That is because in the beginning we assumed that $\valp(a) = 0$. In each of the other cases listed above a similar computation is required to find $\kappa_p^c(j)$. Let
	\begin{equation*}
		\begin{split}
 			\alpha_j(p) 
 			&= 
 			p^{2v_j + 3} - p^{2v_j + 2} + 2 p^{v_j + 3} + 2 p^{v_j + 2} - 4 p^{v_j + 1} - p^3 + p^2 + 2 p - 2, 
 			\\
 			\beta_j(p) 
 			&= 
 			p^{2v_j + 4} - p^{2v_j + 3} + p^{2v_j + 2} - p^{2v_j + 1} + 2 p^{v_j + 4} + p^{v_j + 3} + 3 p^{v_j + 2} 
 			\\
 			&- 5 p^{v_j + 1} - p^{v_j} + 2 p^2 - 2 p, 
 			\\
 			\gamma_j(p)
 			&= 
 			p^{2v_j + 5} + 2 p^{2v_j + 4} + p^{2v_j + 3}.
		\end{split}
	\end{equation*}
	Then the prescribed algorithm above yields
	\begin{equation} \label{kappa_p(v_j)}
		\kappa_p^c(j)
		= 
 		\begin{cases}
 			\frac{3}{4}\alpha_j(p) \( \gamma_j(p) \)^{-1}
 			&\mbox{ if } v_j \mbox { is even},
 			\\
 			\frac{3}{4}\beta_j(p) \( \gamma_j(p) \)^{-1} 
 			&\mbox{ if } v_j \mbox { is odd}.
 		\end{cases}
	\end{equation}
	Now, if we let
	\begin{equation*}
 		\Sigma(p, n) 
 		=
 		\frac{3}{4}
 		\sum_{0 \le j < \valp(n)} \frac{1}{p^{3j}}
 		\( \frac{1 + (-1)^{v_j}}{2} \cdot \frac{\alpha_j(p)}{\gamma_j(p)} + \frac{1 + (-1)^{v_j + 1}}{2} \cdot \frac{\beta_j(p)}{\gamma_j(p)} \),
 	\end{equation*}
	then from \eqref{sigma_p decomp}, \eqref{kappa_p} and \eqref{kappa_p(v_j)} it follows that for each odd prime $p$ we have 
	\begin{equation*}
		\sigma_p 
		= 
		1 - \frac{1}{p^{3\valp(n) + 3}} - \frac{3}{2}(p - 1)p^{-3} - \Sigma(p, n).
	\end{equation*}
	When $p \nmid n$ we have $\Sigma(p, n) = 0$ and then as claimed in Lemma~\ref{LocDens}
	\begin{equation*}
		\sigma_p 
		= 
		1 - \frac{3}{2p^{2}} + \frac{1}{2p^3}.
	\end{equation*}	

	We now consider the case $p = 2$, assuming that $2 \nmid n$. Here the strategy is to look at admissible triples mod $8$. We start with the assumption $\val2(a) = 0$ by symmetry and we look at different cases according to whether $\val2(b)$ and $\val2(c)$ are smaller or greater than $3$ analogous to what we did before. Computing the contribution from each case separately yields
	\begin{equation} \label{kappa_2}
		\sigma_2 
		= 
		\frac{357}{512},
	\end{equation}
	as required.
\end{proof}

We have two remarks. The first one is that \eqref{sigma_p decomp_0} together with the the remark after it, and  \eqref{kappa_p}, \eqref{kappa_2} imply that for each prime we have
\begin{equation*}
	\sigma_p \ge \frac{1}{p^{\valp(n)}} \kappa_p(\valp(n)) 
	> 
	0.
\end{equation*} 
The second remark is that the product $\sigma_\infty \prod_p \sigma_p$ is obviously convergent.
\medskip
  
We continue with the first part of Theorem~\ref{Th0}. It is implied by a result from the work of Bright, Browning and Loughran \cite[Lemma~3.1]{BBL16}. We only need to check that the assumptions of Lemma~3.1 from their work hold in our setting with $\Omega_\infty$ and $\Omega_p$ defined as in \eqref{Omega}.

The positiveness of the measures of $\Omega_\infty$ and $\Omega_p$ for all $p$ has already been shown. It is clear that $\Omega_\infty$ is a union of six unit cubes in $\RR^3$ and therefore its boundary must have zero measure. In the proof of Lemma~\ref{LocDens} we have shown that the complement of $\Omega_p$ in $\Zp^3$ is a finite union of clopen sets. Thus, one concludes that $\Omega_p$ is clopen set which implies that it must have zero measure of the boundary.
	
Let	
\begin{equation*}
	R(B)
	= 
	\# \left\{ (a, b, c) \in B \Omega_\infty \cap \ZZ^3 \ : \ \exists \mbox{ a prime } p > M \mbox{ s.t. } (a, b, c) \not\in \Omega_p \right\}.
\end{equation*}
To check the last hypothesis of \cite[Lemma~3.1]{BBL16} we use the simple fact that when computing the limit
\begin{equation} \label{limMB}
	\lim_{M \to \infty} \limsup_{B \to \infty} \frac{R(B)}{B^3}
\end{equation}	
we can assume that $M > n$ and therefore the prime $p$ appearing in the definition of $R(B)$ does not divide $n$. Therefore, a triple $(a, b, c) \not\in \Omega_p$ if either
\begin{enumerate}
	\item{$p$ divides exactly two of $a, b, c$ and the remaining coefficient times $n$ is not a square mod $p$, or
	}
	\item{$p$ divides all of the coefficients $a, b, c$.
	}
\end{enumerate}
Thus we get
\begin{equation*}
	R(B)
	\ll 
	\# \left\{ (a, b, c) \in B \Omega_\infty \cap \ZZ^3 \ : \ \exists p > M \mbox{ s.t. } p \mid (a, b) \right\}.
\end{equation*}
A simple calculation shows that
\begin{equation*}
	\begin{split}
		R(B)
		&\ll 
		\# \left\{ 0 \le a, b, c \le B \ : \ \exists p > M \mbox{ s.t. } p \mid (a, b) \right\} 
		\\
		&\ll 
		B \sum_{p > M} \# \left\{ 0 < a, b \le B \ : \ p \mid (a, b) \right\} + B^2.
	\end{split}
\end{equation*}
Then since the last sum is bounded above by $B^2 M^{-1}$ we get
\begin{equation*}
	R(B)
	\ll 
	\frac{B^3}{M} + B^2.
\end{equation*}
Therefore the limit given in \eqref{limMB} is equal to zero. This verifies that \cite[Lemma~3.1]{BBL16} is applicable in our situation which concludes the proof of Theorem~\ref{Th0}.

\section{Proof of Theorem~\ref{Th1}} \label{UB}

We begin by establishing Proposition~\ref{Proposition}, following a proof suggested to us by the referee. Let $a$, $b$, $c$, and $n$ be fixed non-zero integers which satisfy one of the conditions $(i)$, $(ii)$ or $(iii)$ in the statement of Proposition~\ref{Proposition}. Let $\X = \X_{a, b, c}$ and $X = \X \times_\ZZ \QQ$. Because of the symmetry in $a$, $b$, and $c$ it suffices to consider the case when there exists a prime $p > 2$ and an odd positive integer $v$ such that
\begin{equation} \label{abcn}
	p^v \dmid a 
	\mbox{ and } 
	p \nmid bcn. 
\end{equation}
We fix $p$ to be one such prime and as before we let $d = -abcn$. We also let $\X_p = \X \times_\ZZ \Zp$ and $X_p = X \times_{\QQ} \Qp$. The hypothesis \eqref{abcn} implies that $d$ is not a square in $\QQ$, nor in $\Qp$. Thus by \cite[\S5.6]{CTX09} we have
\begin{equation*}
	\Br X / \Br \QQ 
	\cong 
	\Br X_p / \Br \Qp 
	\cong
	\ZZ / 2\ZZ.
\end{equation*}

We assume that $\X$ has a $\ZZ_p$-point for all primes including $\infty$. We will show that the map in $\inv_q \circ \ev_{\alpha, q} : \X(\Zp) \rightarrow \{0, 1/2\}$ for a generator $\alpha$ of $\Br X / \Br \QQ$ is surjective at the prime $p$ satisfying \eqref{abcn}. By Remark~\ref{SurjRemark} it is enough to conclude that there is no Brauer--Manin obstruction for $\X$. 

It will suffice to work with a generator $\beta$ of $\Br X_p / \Br \Qp$. We know that $\X_p(\Zp) = \X(\Zp)$. In this case, it will be enough to show that the map
\begin{equation} \label{eq:surj beta}
	\X_p(\Zp) 
	\xrightarrow{\ev_{\beta, p}}  
	\Br \Qp 
	\xrightarrow{\inv_p} 
	\QQ / \ZZ
\end{equation}
takes both values $0$ and $1/2$.

Since $p \nmid bcn$ we have a $\Zp$-point with at least one of the coordinates $y$ or $z$ non-zero and a non-zero $t$-coordinate on the curve $\C \subseteq \PP^2$ given by
\begin{equation*}
	\C: \quad by^2 + cz^2 - nt^2 = 0.
\end{equation*}
If $l_1(y, z, t)$ defines the tangent line to $\C$ at this $\Zp$-point, then there exist linear forms $l_2(y, z, t), l_3(y, z, t) \in \Zp[y, z, t]$ and a non-zero constant $e \in \Qp$ such that
\begin{equation*}
	by^2 + cz^2 - nt^2
	=
	l_1(y, z, t) l_2(y, z, t)
	-
	e^2bcn (l_3(y, z, t))^2.
\end{equation*} 
Thus
\begin{equation*}
	ax^2 + by^2 + cz^2 - nt^2
	=
	ax^2 + l_1(y, z, t) l_2(y, z, t) - e^2bcn (l_3(y, z, t))^2. 
\end{equation*}

Passing to coordinates on $X_p$ we let $l(y, z) = l_1(y, z, t)/t$. Thus, by $\S2$, the quaternion algebra $(l(y, z), -abcn)$ is a representative for the generator of $\Br X_p / \Br \Qp$ and it can be evaluated at each $(x, y, z) \in \X_p(\Zp)$ for which $l(y, z) \neq 0$. Since at least one of the coefficients in front of $y$ or $z$ in $l(y, z)$ is non-zero, there are clearly points in $\X_p(\Zp)$ with $l(y, z) \in \Zp^\times$ equal to a square and points for which it is a non-square unit in $\Zp$. We know that $(l(y, z), -abcn)$ is split over $\Qp$ if and only if the associated Hilbert symbol $(l(y, z), -abcn)_p = 1$. Since the $p$-adic valuation of $abcn$ is odd, when $l(y, z)$ is a unit in $\Zp$ we have
\begin{equation*}
	(l(y, z), -abcn)_p
	=
	\(\frac{l(y, z)}{p}\)
	=
	\begin{cases}
		1 
		&\mbox{if $l(y, z)$ is a square unit in } \Zp,
		\\
		-1
		&\mbox{if $l(y, z)$ is a non-square unit in } \Zp.
	\end{cases}
\end{equation*}
Thus the map \eqref{eq:surj beta} is surjective on $\{0, 1/2\}$, which concludes the proof of Proposition~\ref{Proposition}.
\medskip

We proceed with the proof of Theorem~\ref{Th1}. For any given $k, l, m \in \ZZ$ define $\P_k$ to be the set of odd primes which divide $k$ and do not divide $n$ and define the indicator function
\begin{equation*}
	\delta_n(k, l, m) 
	= 
 	\begin{cases}
 		1  
 		&\mbox{if } p \in \P_k \mbox{ implies either } \valp(k) \mbox{ is even or } p \mid lm, 
 		\\
 		0  
 		&\mbox{otherwise}.
 	\end{cases} 
\end{equation*}
Proposition \ref{Proposition} implies that the surfaces counted by $N_{\Br}(B)$ form a subfamily of the family of quadric surfaces $\X$ which satisfy 
\begin{equation} \label{delta}
	\delta_n (a, b, c) 
	= 
	\delta_n(b, a, c) 
	= 
	\delta_n(c, a, b) 
	= 
	1.
\end{equation}
Bounding from above the number of surfaces in $\mathcal{F}_n$ of height not exceeding $B$ which satisfy \eqref{delta} will allow us to obtain an upper bound for the quantity $N_{\Br}(B)$. This number is equal to
\begin{equation*}
	S(B)
	= 
	\sum_{\substack{a, b, c \in \ZZ \cap [-B, B] \setminus \{ 0 \} \\ a, b, c \text{ not all the same sign}}} \delta_n(a, b, c) \delta_n(b, a, c) \delta_n(c, a, b).
\end{equation*}
Since the signs of $a$, $b$ and $c$ are immaterial, we have
\begin{equation*}
	S(B)
	= 
	6 \sum_{0 < a, b, c \le B} \delta_n(a, b, c) \delta_n(b, a, c) \delta_n(c, a, b).
\end{equation*} 
Any $(a, b, c) \in \NN^3$ for which \eqref{delta} holds admits a factorisation
\begin{equation*}
	\begin{split}
		a 
		&= 
		v_1 u_{12} u_{13} a_1^2, 
		\\
		b 
		&= 
		v_2 u_{12} u_{23} b_1^2, \\
		c
		&= 
		v_3 u_{13} u_{23} c_1^2.
	\end{split}
\end{equation*}
The factors $v_i$, $i = 1, 2, 3$ consist of powers of the primes in $\Q = \{2\} \cup \{p \mid n\}$. Let 
\begin{equation*}
	\ve \( v \) 
	=
 	\begin{cases} 
 		1 
 		&\mbox{if } p \mid v \Rightarrow p \in \Q, 
 		\\
 		0 
 		&\mbox{otherwise}.
 	\end{cases}
\end{equation*}
Since there are $O \( B^{1/2} \)$ squares not exceeding $B$, if we sum over $a_1$, $b_1$ and $c_1$ first we get
\begin{equation*}
	S(B) 
	\ll 
	B^{\frac{3}{2}} \sum_{i = 1}^3 
	\sum_{v_i \le B} \frac{\ve(v_i)}{\sqrt{v_i}} 
	\sum_{u_{12}, u_{13}, u_{23} \le B} \frac{1}{u_{12} u_{13} u_{23}}.
\end{equation*}
Next summing over $u_{12}$, $u_{13}$ and $u_{23}$ gives 
\begin{equation*}
	S(B) 
	\ll 
	B^{\frac{3}{2}} \( \log B \)^3
	\sum_{i = 1}^3 
	\sum_{v_i \le B} \frac{\ve(v_i)}{\sqrt{v_i}}.
\end{equation*}
Over $v_i$ we are summing non-negative integers thus if we complete those sums we get something even bigger. It is clear that $\ve \( \cdot \)$ is completely multiplicative function which allows us to write the completed sums as Euler products and therefore we get
\begin{equation*}
	S(B) 
	\ll  
	B^{\frac{3}{2}}  \( \log B \)^3
	\prod_{p \mid 2n} \( 1 - \frac{1}{p^{\frac{1}{2}}} \)^{-3} 
	\ll_n  
	B^{\frac{3}{2}} \( \log B \)^3.
\end{equation*}
This completes the proof of Theorem~\ref{Th1}.

\section{Proof of Theorem~\ref{Th2}} \label{AF}
	
We begin the proof of Theorem~\ref{Th2} by studying a more general family than $\mathcal{F}_1'$, and in particular the conditions under which each surface in this family has points everywhere locally but has no points globally. For any non-zero $a$, $b$, $c \in \ZZ$ let $\X_{a, b, c}^{\ast} \subseteq \AA^3$ be the variety over $\ZZ$ defined by
\begin{equation} \label{X ast}
	\X_{a, b, c}^\ast : \quad a x^2 + b y^2 + c^2 z^2 = 1
\end{equation}	
with $a$, $b$ not both positive and let $X_{a, b, c}^\ast = \X_{a, b, c}^\ast \times_\ZZ \QQ$. Clearly, $X_{a, b, c}^\ast(\RR) \neq \emptyset$.  We start with a simple lemma concerning the local solubility of $\X_{a, b, c}^\ast$.

\begin{lemma} \label{local X'}
	Let $p$ be an odd prime. Then $\X^\ast_{a, b, c}(\Zp) \neq \emptyset$ precisely when $p \nmid (a, b, c)$ and the following two conditions hold:
 	\begin{enumerate}
 		\item[\emph{(i)}]{if $p \mid (a, c)$, then $b$ is a quadratic residue mod $p$,
 		}
 		\item[\emph{(ii)}]{if $p \mid (b, c)$, then $a$ is a quadratic residue mod $p$.
 		}	
 	\end{enumerate}
\end{lemma}

\begin{proof}
	Suppose $p \mid (a, b, c)$, then obviously $\X^\ast_{a, b, c}(\Zp) = \emptyset$. On the contrary, if $p \nmid abc$ or $p$ divides exactly one of the coefficients $a$, $b$, $c^2$, then counting the possible values of $x$, $y$, $z$ mod $p$ shows that the reduced equation defining $\X^\ast_{a, b, c}(\Zp)$ mod $p$ is soluble. Since $p$ is an odd prime, then Hensel's lemma is applicable and therefore $\X^\ast_{a, b, c}(\Zp) \neq \emptyset$. If $p \mid (a, b)$ then again reduction mod $p$ and Hensel's lemma prove the existence of a point in $\X^\ast_{a, b, c}(\Zp)$.

	Because of the symmetry in $a$ and $b$ in $\X^\ast_{a, b, c}$, in the remaining case it suffices to consider that $p \mid (a, c)$ and $p \nmid b$. Assume that $b$ is square mod $p$. Then, as above, we can use Hensel's lemma to show that $\X^\ast_{a, b, c}(\Zp) \neq \emptyset$. On the other hand, the existence of a point in $\X^\ast_{a, b, c}(\Zp)$ and the reduction of $a x^2 + b y^2 + c^2 z^2 = 1$ mod $p$ shows that $b$ is a quadratic residue mod $p$.
\end{proof}

We proceed with the explicit conditions under which there is an integral Brauer--Manin obstruction for $\X^\ast_{a, b, c}$. We have the rational point $(0, 0 , -1/c) \in X^\ast_{a, b, c} \( \QQ \)$ and the algorithm explained in $\S\ref{IBM}$ gives a non-trivial generator $(1 + cz, -abc^2) = (1 + cz, -ab)$ of $\Br X^\ast_{a, b, c} / \Br \QQ$ when $-ab$ is not a square in $\QQ$. Define $U$ by 
\begin{equation} \label{defU}
	U: \quad ax^2 + by^2 = (1 - cz)(1 + cz) \neq 0. 
\end{equation}
Following \cite[\S 8.2]{CTX09}, we study the behaviour of the generator $(1 + cz, -ab)$ when $(x, y ,z)$ is a point in the set $\X^\ast_{a, b, c}(\Zp) \cap U(\Qp)$. This is sufficient since $\X^\ast_{a, b, c}(\Zp) \cap U(\Qp)$ is a dense subset of $\X^\ast_{a, b, c}(\Zp)$. Before we proceed we would like to make two important remarks. The first one is that for any point $(x, y ,z) \in \X^\ast_{a, b, c}(\Zp) \cap U(\Qp)$ \eqref{defU} yields
\begin{equation} \label{qalgebra swap}
	(1 + cz, -ab) 
	= 
	(a(1 - cz), -ab)
\end{equation}
as quaternion algebras. The second one is that since
\begin{equation} \label{zero valuation}
	(1 + cz) + (1 - cz) 
	= 
	2,
\end{equation}
then for each odd prime $p$ the $p$-adic valuation of at least one of $(1 + cz)$ or $(1 - cz)$ is zero. 

We look at the case when $p$ is an odd prime first. We separate the study into different cases obtained under divisibility conditions. 
\begin{enumerate}
	\item[---]{If $p \nmid ab$, then
	\begin{equation*}
		\begin{split}
			(1 + cz, -ab)_p 
			&= 
			\( \frac{-ab}{p} \)^{\valp(1 + cz)}, 
			\\
			(a(1 - cz), -ab)_p 
			&= 
			\( \frac{-ab}{p} \)^{\valp(a) + \valp(1 - cz)}.
		\end{split}
	\end{equation*}
	Since $\valp(a) = 0$, \eqref{qalgebra swap} and \eqref{zero valuation} imply that 
	\begin{equation*}
		(1 + cz, -ab)_p 
		= 
		(a(1 - cz), -ab)_p 
		= 
		1.
	\end{equation*}
	}	
	\item[---]{If $p \mid c$ then 
	\begin{equation*}
		(1 + cz, -ab)_p 
		= 
		(1, -ab)_p 
		= 
		1.
	\end{equation*}
	}
	\item[---]{If $p \nmid c$ and $\valp(ab) \equiv 0$ mod $2$ then 
	\begin{equation*}
		(1 + cz, -ab)_p 
		= 
		(1 + cz, -abp^{-\valp(ab)})_p.
	\end{equation*}
	Similar argument as in the case when $p \nmid ab$ implies that if $-abp^{-\valp(ab)}$ is a square mod $p$ or $\valp(a)$ is even, then 
	\begin{equation*}
		(1 + cz, -abp^{-\valp(ab)})_p 
		= 
		1.
	\end{equation*}
	If $-abp^{-\valp(ab)}$ is not a square mod $p$ and $\valp(a)$ is odd, then
	\begin{equation*}
		(1 + cz, -ab)_p
		= 
		(-1)^{\valp(1 + cz)}
		= 
		(-1)^{1 + \valp(1 - cz)}.
	\end{equation*}
	For all $(x, y, z) \in \X^\ast_{a, b, c}(\Zp) \cap U(\Qp)$ we have either $z \equiv 1/c$ or $z \equiv -1/c$ mod $p$. Thus, by Remark~\ref{SurjRemark} there is no integral Brauer--Manin obstruction in this case.  
	}
	\item[---]{If $p \nmid c$, $p \mid (a, b)$ with $\valp(ab) \equiv 1$ mod $2$, then the reduction of $\X^\ast_{a, b, c}$ mod $p$ forces all $(x, y, z)$ to satisfy $z \equiv \pm 1 / c$ mod $p$. Assume that $\valp(a) \equiv 0$ mod $2$. If $z \equiv 1 / c$ mod $p$, then we have 
	\begin{equation} \label{2p}
		(1 + cz, - ab)_p 
		= 
		\( \frac{2}{p} \).
	\end{equation}
	On the other hand, when $z \equiv -1 / c$ mod $p$ then 
	\begin{equation*}
	(a(1 - cz), - ab)_p 
	= 
	\( \frac{2ap^{-\valp(a)}}{p} \).
	\end{equation*}
	Thus, by \eqref{qalgebra swap} we conclude that \eqref{2p} holds to all $(x, y, z) \in \X^\ast_{a, b, c}(\Zp) \cap U(\Qp)$ unless $ap^{-\valp(a)}$ is not a square mod $p$. In the latter case, the map of \eqref{eq:ev composed with inv} which sends $(x, y, z)$ to $\inv_p(1 + cz, - ab)$ takes both values $0$ and $1/2$. Therefore, Remark~\ref{SurjRemark} implies that there is no integral Brauer--Manin obstruction. 	
	
	When $\valp(b)$ is even, then
	\begin{equation*}
		(a(1 - cz), - ab)_p 
		= 
		\( \frac{2bp^{-\valp(b)}}{p} \),
	\end{equation*}
	and the same observation can be made about the residue class mod $p$ of $bp^{-\valp(b)}$.
	}	
	\item[---]{If $p \nmid bc$ and $\valp (a) \equiv 1$ mod $2$, then Proposition~\ref{Proposition} guarantees that there is no integral Brauer--Manin obstruction provided $p \neq 2$.
 	}	
\end{enumerate}

When $p = 2$ one can study which classes of triples $(a, b, c)$ mod 8 are admissible. It is easy to see that if $4 \mid c$ and $2 \nmid ab$ then 
\begin{equation*}
	(1 + cz, - ab)_2 
	= 
	1.
\end{equation*}
	
Finally, when $a$ and $b$ have different signs, then $-ab > 0$ and $(1 + cz, - ab)$ is split over $X^\ast_{a, b, c}(\RR)$. On the contrary, if both $a$ and $b$ are negative, then the existence of points $(x, y, z) \in X^\ast_{a, b, c}(\RR)$ for which $1 + cz$ takes both positive and negative values and Remark~\ref{SurjRemark} imply that there is no integral Brauer--Manin obstruction in this case.

Now one can understand why we restrict our attention to the surfaces $\X_{a, b, c}'$ given in \eqref{X'}. Lemma~\ref{X ast} and the condition $a \equiv 1$ mod 8 justify that the space of ad\`{e}les of $\X'_{a, b, c}$ is non-empty. Recall that in the definition of $\X'_{a, b, c}$ we require $a$, $b$, $c$ to be positive integers not divisible by 3 with $a$, $b$ squarefree, $a \equiv 1$ mod 8 and $(ab, 2c) = 1$. The discussion above together with Proposition~\ref{Proposition} leads to the conclusion that there is an integral Brauer--Manin obstruction for $\X'_{a, b, c}$ precisely when $a = b \equiv 1$ mod $24$ and if a prime $p$ divides $a$, then $3$ is a square mod $p$. Rewriting $a = u$ and $c = v$, we therefore find that
\begin{equation*}
	N'_{\Br}(B)
	=
	\sum_{\substack{u \le B/9 \\ u \equiv 1 \text{ mod } 24}} \mu^2(u)
	\prod_{p \mid u}\frac{1}{2} \( 1 + \( \frac{3}{p} \)\)
	\sum_{\substack{v \le \sqrt{B}/4 \\ (v, 3u) = 1}} 1.
\end{equation*}
We begin with the sum over $v$. Using a well-known identity we replace the condition $(v, 3u) = 1$ by a sum of the M\"{o}bius function $\mu(d)$ over the divisors $d$ of $(v, 3u)$. Then swapping the order of summation gives 
\begin{equation*}
	\sum_{\substack{v \le \sqrt{B}/4 \\ (v, 3u) = 1}} 1
	=
	\frac{B^{\frac{1}{2}}}{6}
	\sum_{d \mid u} \frac{\mu(d)}{d} 
	+ 
	O(\tau(u)).
\end{equation*}
We apply the above asymptotic formula for the sum over $v$ in $N'_{\Br}(B)$ to get
\begin{equation*}
	N'_{\Br}(B)
	=
	\frac{B^{\frac{1}{2}}}{6}
	\sum_{\substack{u \le B/9 \\ u \equiv 1 \text{ mod } 24}} \frac{\mu^2(u)}{\tau(u)}
	\prod_{p \mid u} \( 1 + \( \frac{3}{p} \)\)
	\sum_{d \mid u} \frac{\mu(d)}{d}
	+ 
	O(B \log B).
\end{equation*}
We encode the condition $u \equiv 1$ mod $24$ in a standard way by a sum over the multiplicative characters mod $24$. Then we change the order of summation over $u$ and $d$ and we write the sum over $u$ as a double sum. We let $a_{\chi}(d) = \mu\(d\) \chi\(d\) \prod_{p \mid d}\( 1 + \( \frac{3}{p} \)\)$. Then
\begin{equation*}
	N'_{\Br}(B)
	=
	\frac{B^{\frac{1}{2}}}{48}
	\sum_{\chi \text{ mod } 24 \text{ }}
	\sum_{d \le B/9} \frac{a_{\chi}(d)}{\tau\(d\) d}
	\sum_{\substack{u'u'' \le B/9d \\ (u'u'', d) = 1}} \frac{\mu^2(u'u'')}{\tau(u'u'')}\chi\(u'u''\)\(\frac{3}{u''}\) 
	+ 
	O(B \log B).
\end{equation*}

Let $V_{\chi}(B)$ denote the inside sum over $d$, $u'$ and $u''$, and let $\chi_1\(\ast\)$ denote the multiplicative character $\(\frac{3}{\ast}\) \chi\(\ast\)$ mod $24$, so that
\begin{equation} \label{N_br^5 split}
	N'_{\Br}(B)
	=
	\frac{B^{\frac{1}{2}}}{48}
	\sum_{\chi \text{ mod } 24} V_{\chi}(B)
	+
	O(B \log B).
\end{equation}
We split $V_{\chi}(B)$ into two sums according to the range of $d$. In the first one $d \le B^\frac{1}{2}$ and the second one $B^\frac{1}{2} < d \le B/9$. Using $a_{\chi}(d) \ll \tau(d)$, we bound trivially the sum with $d$ big by $B^\frac{1}{2} \log B$. Then we further split the sum with $d$ small into three sums with respect to the values of $u'$ and $u''$ compared to each other. Thus we get 
\begin{equation} \label{V_chi}
	V_{\chi}(B)
	=
	\sum_{d \le \sqrt{B}} \frac{a_{\chi}(d)}{\tau\(d\) d} \sum_{i = 1}^{3} W_{\chi}^i (B, d)
	+
	O\( B^\frac{1}{2} \log B \),
\end{equation}
where
\begin{equation*}
	W_{\chi}^i (B, d)
	=
	\sum_{\substack{u' \le B/9d \\ (u', d) = 1}} \frac{\mu^2\(u'\)}{\tau\(u'\) } \chi\(u'\)
	\sum_{u'' \in E_i} \frac{\mu^2\(u''\)}{\tau\(u''\)} \chi_1\(u''\)
\end{equation*}
and for $i = 1, 2, 3$ the sets $E_i$ are defined as follows:
\begin{equation*}
	\begin{split}
		E_1
		&=
		\{ u'' \in \NN \ : \ (u'', u'd) = 1, u'' > u' \text{ and } u'' \le B/9u'd\},
		\\
		E_2
		&=
		\{ u'' \in \NN \ : \ (u'', u'd) = 1, u'' = u' \text{ and }  u'' \le B/9u'd \}, 
		\\
		E_3
		&=
		\{ u'' \in \NN \ : \ (u'', u'd) = 1, u'' < u' \text{ and }  u'' \le B/9u'd\}.
	\end{split}
\end{equation*}
	
For simplicity when $i = 1, 2, 3$ let
\begin{equation} \label{V_chi^i}
	V_{\chi}^i (B)
	=
	\sum_{d \le \sqrt{B}} \frac{a_{\chi}(d)}{\tau\(d\) d} W_{\chi}^i (B, d).
\end{equation}
Clearly, $(u', u'') = 1$ and $u' = u''$ are simultaneously satisfied only when $u' = u'' = 1$, thus
\begin{equation} \label{V_chi^2}
	V_{\chi}^2(B)
	\ll 
	\log B.
\end{equation}
	
By symmetry, we can treat only one of $W_{\chi}^1 (B, d)$ and $W_{\chi}^3 (B, d)$. For this reason, from now on we will be concerned with $W_{\chi}^1 (B, d)$ only. From the condition $u' < u'' \le B/9du'$ we observe that the summation in $W_{\chi}^1 (B, d)$ over $u'$ must be in the range $1 \le u' < \sqrt{B/9d}$. We write $W_{\chi}^1 (B, d)$ as the difference of two sums in the following way:
\begin{equation*}
	\begin{split}
		W_{\chi}^1 (B, d)
		&=
		\sum_{\substack{u' < \sqrt{B/9d} \\ (u', d) = 1}}\frac{\mu^2(u')}{\tau(u') } \chi\(u'\)
		\sum_{\substack{u'' \le B/9u'd \\ (u'', u'd) = 1}}\frac{\mu^2(u'')}{\tau(u'')} \chi_1\(u''\) 
		\\
		&-
		\sum_{\substack{u' < \sqrt{B/9d} \\ (u', d) = 1}}\frac{\mu^2(u')}{\tau(u') } \chi\(u'\)
		\sum_{\substack{u'' \le u' \\ (u'', u'd) = 1}}\frac{\mu^2(u'')}{\tau(u'')} \chi_1\(u''\). 
	\end{split}
\end{equation*}
We use the well-known upper bound $\sum_{n \le X} (\tau(n))^{-1} \ll X (\log X)^{-\frac{1}{2}}$ for the latter sum to obtain
\begin{equation} \label{W1 error}
	W_{\chi}^1 (B, d)
	=
	\sum_{\substack{u' < \sqrt{B/9d} \\ (u', d) = 1}}\frac{\mu^2(u')}{\tau(u') } \chi\(u'\)
	\sum_{\substack{u'' \le B/9u'd \\ (u'', u'd) = 1}}\frac{\mu^2(u'')}{\tau(u'')} \chi_1\(u''\)
	+
	O\(\frac{1}{d} \cdot \frac{B}{\log B}\).
\end{equation}
If we let
\begin{equation*}
	S_{\chi}^1(B)
	=
	\sum_{d \le \sqrt{B}}\frac{a_{\chi}(d)}{\tau\(d\) d} 
	\sum_{\substack{u' < \sqrt{B/9d} \\ (u', d) = 1}} \frac{\mu^2\(u'\)}{\tau\(u'\) } \chi\(u'\)
	\sum_{\substack{u'' \le B/9u'd \\ (u'', u'd) = 1}} \frac{\mu^2\(u''\)}{\tau\(u''\)} \chi_1\(u''\),
\end{equation*}
then \eqref{V_chi^i} and \eqref{W1 error} imply 
\begin{equation} \label{V_chi S_chi}
	V_{\chi}^1 (B)
	=
	S_{\chi}^1(B) + O\( \frac{B}{\log B} \).
\end{equation}
	
We are now in position to apply \cite[Lemma~1]{FI10} to the sum over $u''$ in $S_{\chi}^1(B)$. We choose $C = 2$. In the case when $\chi_1$ is the principal character mod $24$ and therefore $\chi$ is the extension of $\(\frac{3}{\ast}\)$ mod $24$ we get
\begin{equation} \label{S_chi^1}
	S_{\chi}^1(B)
	=
	\frac{8B}{105 \pi^\frac{1}{2}}
	\prod_{p}\( 1 + \frac{1}{2p} \)\( 1 - \frac{1}{p} \)^{\frac{1}{2}}
 	T_1(B) + O\( \frac{B \( \log\log B \)^{\frac{3}{2}}}{\log B} \),
\end{equation}	 
where
\begin{equation} \label{T_1}
	\begin{split}
		T_1(B)
		&=
		\sum_{\substack{d \le \sqrt{B} \\ (d, 2) = 1}}\frac{\mu\(d\)}{\tau\(d\) d^2}\(\frac{3}{d}\)
		\prod_{p \mid d}\( 1 + \( \frac{3}{p} \)\) \( 1 + \frac{1}{2p} \)^{-1} 
		\\
		&\times
		\sum_{\substack{u' < \sqrt{B/9d} \\ (u', 2d) = 1}} \frac{\mu^2(u')}{\tau(u') u'} \(\frac{3}{u'}\)
		\prod_{p \mid u'} \( 1 + \frac{1}{2p} \)^{-1} \(\log\(\frac{B}{9u'd}\)\)^{-1/2}.
	\end{split}
\end{equation}
On the other hand, when $\chi_1$ is not the principal character mod $24$, then
\begin{equation} \label{S_chi error}
	S_{\chi}^1(B)
	\ll
	\frac{B}{\log B}.
\end{equation}
It remains to find the asymptotic behaviour of $T_1(B)$. To do so, we begin by studying
\begin{equation}\label{U, U'}
	U_1(B, d)
	=
	\sum_{\substack{u' \le \sqrt{B/9d} \\ (u', 2d) = 1}} \frac{\mu^2(u')}{\tau(u') u'} \( \frac{3}{u'} \)
	\prod_{p \mid u'} \( 1 + \frac{1}{2p} \)^{-1}.
\end{equation}

\begin{lemma} \label{U'_1 af}
	Let $\chi$ be the extension of the quadratic character $\( \frac{3}{\ast} \)$ mod 24. Then the following asymptotic formula holds
	\begin{equation*}
		U_1(B, d)
		=
		\prod_{p} \(1 + \frac{\chi(p)}{2p + 1} \)
		\prod_{p \mid d} \(1 + \frac{\chi(p)}{2p + 1} \)^{-1}
		+
		O \( \tau(d)\frac{\( \log B \)^\frac{1}{2}}{B^{\frac{1}{36}}} \).
	\end{equation*}
\end{lemma}

\begin{proof}
	Let $F(s)$ denote the Dirichlet series for $U_1(B, d)$; that is
	\begin{equation*}
		F(s)
		=
		\sum_{\substack{k = 1 \\ (k, d) = 1}}^{\infty} \frac{\mu^2(k)}{\tau(k) k^{s + 1}} \chi(k)
		\prod_{p \mid k} \( 1 + \frac{1}{2p} \)^{-1}.
	\end{equation*}
	It is clear that $F(s)$ is an absolutely convergent series for $\re(s) > 0$. By Perron's formula \cite[Lemma~3.12]{Tit87} we have
	\begin{equation*}
		U_1(B, d)
		=
		\frac{1}{2 \pi i}
		\int\limits_{1 - i T}^{1 + i T} F(s)\frac{(B/9d)^{s/2}}{s} ds 
		+
		O\(\frac{1}{T} \( \frac{B}{d} \)^{\frac{1}{2}}\).
	\end{equation*}
	Observing that $F(s)$ has analytic continuation to $\re(s) > -1/2$, which is given by
	\begin{equation*}
		F(s)
		=
		\prod_{p} \(1 + \frac{\chi(p)}{(2p + 1)p^s} \) 
		\(1 - \frac{\chi(p)}{p^{s + 1}}\)^{1/2}	
		\prod_{p \mid d} \(1 + \frac{\chi(p)}{(2p + 1)p^s} \)^{-1}
		L\(s + 1, \chi\)^{1/2},
	\end{equation*}
	allows us to move the contour of integration to the left, encountering a simple pole at $s = 0$. More precisely, for some $0 < \eta < 1/2$ which will be chosen suitably later and for $i = 1, 2, 3$ we let $\gamma_i$ be the complex lines from $1 - iT$ to $-\eta - iT$, from $-\eta - iT$ to $-\eta + iT$, and from $-\eta + iT$ to $1 + iT$, respectively. Then, by the Residue theorem we have
	\begin{equation*}
		U_1(B, d)
		=
		\underset{s = 0}{\Res}\(F(s)\frac{(B/9d)^{s/2}}{s} \) + \frac{1}{2 \pi i} \sum_{i = 1}^3 
		\int\limits_{\gamma_i} F(s) \frac{(B/9d)^{s/2}}{s} ds 
		+
		O\(\frac{1}{T} \( \frac{B}{d} \)^{\frac{1}{2}}\).
	\end{equation*}

	We need to bound the integrals over each $\gamma_i$. We use the classical notation $s = \sigma + it$. In the estimation of each of the integrals we shall use the fact that when $-\eta \le \sigma \le 1$ the finite product in the analytic continuation of $F(s)$ is bounded by $\tau(d)$, the infinite product is convergent and we have the following bound for $L(s, \chi)$ (see \cite{Kol79}):
	\begin{equation*}
		|L(s, \chi)| \ll_\ve (|t| + 2)^{\frac{1}{6}(1 -\sigma) + \ve}
		\quad \mbox{when} \quad
		\frac{1}{2} \le \sigma \le 1.
	\end{equation*} 
	Therefore, if $T$ is big enough we have
	\begin{equation*}	
		\max\limits_{-\eta \le \sigma \le 1}|F(s)| 
		\ll_\ve 
		\tau(d) T^{\frac{\eta}{6} + \ve}.
	\end{equation*}
	The above implies
	\begin{equation*}
		\left| \int\limits_{\gamma_1} \right|, 
		\left| \int\limits_{\gamma_3} \right|
		\ll_\ve 
		\tau(d) T^{\frac{\eta}{6} + \ve - 1} \sqrt{\frac{B}{d}}
	\end{equation*}
	and
	\begin{equation*}
		\left| \int\limits_{\gamma_2} \right|
		\ll_{\ve, \eta} 
		\tau(d) T^{\frac{\eta}{6} + \ve} \( \frac{d}{B} \)^{\frac{\eta}{2}}
		\int\limits_{-T}^{T} \frac{1}{1 + |t|} dt
		\ll_{\ve, \eta} 
		\tau(d) T^{\frac{\eta}{6} + \ve} \log T \( \frac{d}{B} \)^{\frac{\eta}{2}}.
	\end{equation*}
	Choosing $T = \frac{B}{d}$, $\eta  = \frac{1}{3}$ and $\ve = \frac{1}{18}$ and taking into account that $d \le B^\frac{1}{2}$ gives
	\begin{equation*}
		\max_{i = 1, 2, 3}
		\left| \int\limits_{\gamma_i} \right|
		\ll
		\tau(d) B^{-\frac{1}{36}} \log B.
	\end{equation*}

	It remains to compute the residue of $F(s)\frac{(B/9d)^{s/2}}{s}$ at $s = 0$. It is easy to see that this residue is equal to $F(0)$. This concludes the proof of the lemma.
\end{proof}

Until the end of the proof of the asymptotic formula for $V_{\chi}^1(B)$ we shall fix $\chi$ to be the extension of $\(\frac{3}{\ast}\)$ mod $24$. Recall the definition of $U_1(B, d)$ in \eqref{U, U'} and the expression for $T_1(B)$ in \eqref{T_1}. We apply Lemma~\ref{U'_1 af} with partial summation to conclude that 
\begin{equation} \label{T_1 af}
	T_1(B)
	=
	\frac{1}{\(\log \frac{4B}{9}\)^{\frac{1}{2}}}
	\prod_{p} \(1 + \frac{\chi(p)}{2p + 1} \) \(1 - \frac{\chi(p) + \chi(p)^2}{p(2p + 1 + \chi(p))} \)
	+
	O \(\frac{\( \log B \)^\frac{1}{2}}{B^{\frac{1}{36}}} \).
\end{equation}	
What is left is to put together \eqref{V_chi S_chi}, \eqref{S_chi^1}, \eqref{S_chi error} and \eqref{T_1 af} with the binomial theorem to get
\begin{equation*}
	V_{\chi}^1(B)
	=
	C_1 \frac{B}{\(\log B\)^{\frac{1}{2}}}
	+ 
	O\( \frac{B \( \log\log B \)^{\frac{3}{2}}}{\log B} \),
\end{equation*}
where
\begin{equation*}
	C_1
	=
	\frac{4}{45 \pi^\frac{1}{2}}
	\prod_{p}\( 1 + \frac{1}{2p} \)\( 1 - \frac{1}{p} \)^{1/2} \(1 + \frac{\chi(p)}{2p + 1} \) \(1 - \frac{1 + \chi(p)}{p(2p + 1 + \chi(p))} \) .
\end{equation*}
The same argument allows us to establish the same asymptotic formula for $V_{\chi}^3(B)$ with $C_1$ being the constant in the leading term. It remains to combine \eqref{N_br^5 split}, \eqref{V_chi}, \eqref{V_chi^i} and \eqref{V_chi^2} with the asymptotic formulae for $V_{\chi}^1(B)$ and $V_{\chi}^3(B)$. In this way we obtain
\begin{equation*}
	N'_{\Br}(B)
	=
	\frac{C_1}{24} \cdot
	\frac{B^\frac{3}{2}}{\(\log B\)^{\frac{1}{2}}}
	+
	O\( \frac{B^\frac{3}{2} \( \log\log B \)^{\frac{3}{2}}}{\log B}\).
\end{equation*}
This concludes the proof of Theorem~\ref{Th2}.

\nocite{*}
\bibliographystyle{alpha}
\bibliography{bibliography/references}
\end{document}